\documentclass[12pt]{amsart}
\usepackage{amsmath,amssymb,times,color}
\usepackage{graphicx}
\usepackage{fig4tex}
\usepackage{dsfont}
\definecolor{webred}{rgb}{0.75,0,0}
\definecolor{webgreen}{rgb}{0,0.75,0}
\usepackage[citecolor=webgreen,colorlinks=true,linkcolor=webred]{hyperref}
\usepackage{stmaryrd} 
\usepackage{mathrsfs}
\usepackage{fig4tex}

\newtheorem{thm}{Theorem}[section]
\newtheorem{cor}[thm]{Corollary}
\newtheorem{lem}[thm]{Lemma}

\newtheorem{hyp}[thm]{Assumption}

\theoremstyle{definition}

\theoremstyle{remark}
\newtheorem{rem}[thm]{Remark}

\oddsidemargin=\evensidemargin

\numberwithin{equation}{section}

\renewcommand{\Re}{\operatorname{Re}}

\renewcommand{\Im}{\operatorname{Im}}

\newcommand{\Div}{\operatorname{\mathrm{div}}}

\newcommand{\rot}{\operatorname{\mathrm{curl}}}

\newcommand{\db}[1]{_{\raise-0.3ex\hbox{$\scriptstyle #1$}}}
\newcommand{\dd}[1]{_{\raise-1.5pt\hbox{$\scriptstyle #1$}}}

\newcommand{\di}{\displaystyle}

\newcommand{\dr}{{\rm d}}

\newcommand{\iso}{_{\sf +}}
\newcommand{\con}{_{\sf -}}

\newcommand  {\C}{{\mathbb C}}
\newcommand  {\N}{{\mathbb N}}

\newcommand  {\R}{{\mathbb R}}

\newcommand {\Id}{\mathbb {I}}

\renewcommand  {\H}{{\mathrm H}}

\newcommand{\BB}{\boldsymbol{\mathsf B}}
\newcommand  {\EE}{\boldsymbol{\mathsf E}}

\newcommand  {\HH}{\boldsymbol{\mathsf H}}
\newcommand  {\KK}{\boldsymbol{\mathsf K}}
\newcommand  {\LL}{\boldsymbol{\mathsf L}}
\renewcommand  {\L}{{\mathrm L}}
\newcommand  {\NN}{\boldsymbol{\mathsf N}}

\newcommand  {\RR}{\boldsymbol{\mathsf R}}

\newcommand  {\VV}{\,\underline{\!{\boldsymbol{\mathfrak H}}\!}\,}

\newcommand  {\XX}{\boldsymbol{\mathsf X}}

\newcommand  {\mur}{{\boldsymbol{\mu_{r}}}}

\newcommand  {\g}{{\boldsymbol{\mathsf g}}}

\newcommand  {\jj}{{\boldsymbol{\mathsf j}}}
\newcommand  {\nn}{\boldsymbol{\mathsf n}}

\newcommand  {\uu}{\boldsymbol{\mathsf u}}
\newcommand  {\vv}{\boldsymbol{\mathsf v}}

\newcommand  {\xx}{\boldsymbol{\mathsf x}}
\newcommand  {\yy}{\boldsymbol{\mathsf y}}

\newcommand  {\cH}{\mathcal{H}}
\newcommand  {\cC}{\mathcal{C}}

\newcommand  {\cU}{{\mathcal U}}

\newcommand  {\cV}{{\mathfrak H}}

\newcommand  {\bH}{\mathbf{H}}
\newcommand  {\bL}{\mathbf{L}}

\newcommand  {\sj}{\mathsf{j}}
\newcommand  {\eps}{\mathsf{\varepsilon}}

\newcommand  {\sv}{\mathfrak{h}}

\newcommand  {\bs}{\underline \sigma}
\newcommand  {\bm}{\underline \mu}

\def\on#1{\!\left.\vphantom{|_|}\right|_{#1}}  %Restriktion

\definecolor{mpurple}{rgb}{0.6,0,0.8}
\definecolor{myblue}{rgb}{0.,0.2,0.8}
\definecolor{mygreen}{rgb}{0,0.75,0.0}
\definecolor{mred}{rgb}{0.9,0,0}
\definecolor{mbrun}{rgb}{0.8,0.5,0}

\newcommand{\Bk}{\color{black}}

\begin{document}

\title[Uniform estimates for transmission problems in electromagnetism
]{Uniform estimates for transmission problems in electromagnetism with high contrast in magnetic permeabilities 
}

\author{Victor P\'eron}
\address{Laboratoire de Math\'ematiques et de leurs Applications de Pau (UMR CNRS 5142), E2S UPPA, CNRS, Universit\'e de Pau et des Pays de l'Adour, 64013 Pau cedex, France, victor.peron@univ-pau.fr
}

\begin{abstract}
 We consider the time-harmonic Maxwell equations set on a domain made up of two subdomains that represent a  magnetic conductor and a non-magnetic material, and we assume that the relative magnetic permeability  $\mu_{r}$  between the two materials is high.  We prove uniform a priori estimates for Maxwell solutions when the interface between the two subdomains is supposed to be Lipschitz.  
 Assuming smoothness for the interface between the subdomains, we prove also that the magnetic field possesses a multiscale expansion in powers of  $\eps=1/ \sqrt{\mu_{r}}$ with {\em profiles} rapidly decaying inside the magnetic conductor. 
\end{abstract}
\date{\today, Version 1}

\maketitle

\vskip0.5cm

\section{Introduction}

The numerical computation of the electromagnetic field in strongly heterogeneous media is crucial in many subject areas in real-world applications, see, {\it e.g.},  \cite{eskola2001measurement,PardoDTP06,HJN08,mukherjee20113d,CDFP11,noh2016analysis,ren2016homogenization,issa2019boundary,Pe19,ammari2020superresolution,kawaguchi2022time,ammari2023mathematical,AP23} where additional references
may also be found. The computation can be made particularly difficult due in particular to very high contrasts of magnetic permeabilities, see, {\it e.g.}, \cite{eskola2001measurement,mukherjee20113d,noh2016analysis,ren2016homogenization,AP23}. 
In several fields of applied geophysics, for instance, the computation of the controlled source electromagnetic induction (CSEM) response of conductive and permeable bodies buried in a conductive and permeable medium is widely used, but the presence of buried objects, such as steel pipes, concrete rebar, cables, drums, unexploded ordnance, or storage tanks introduce significant complexity into the CSEM response, especially since these objects often present significant relative magnetic permeability, see, {\it e.g.}, \cite{eskola2001measurement,mukherjee20113d,noh2016analysis}.
High contrasts in magnetic permeabilities can also pose a significant challenge for developing homogenization techniques to deduce effective properties of heterogeneous materials. These techniques are particularly necessary for the design of optimal electromagnetic devices based on soft magnetic composites in many electrical engineering applications, see, {\it e.g.}, \cite{ren2016homogenization}.

 Many works in the literature are  related to issues of uniform estimates for stiff transmission problems in electromagnetism. We refer the reader, for instance, on scattering problems of an electromagnetic plane wave by a  medium  of high index (i.e., with high complex permeability, and high complex permittivity) and of class $\mathcal{C}^2$ (smooth),  \cite{artola1991diffraction,artola1992diffraction}.  In particular,   the authors provided a rigorous proof of Leontovich conditions and proved accurate error estimates in Ref.  \cite{artola1992diffraction}, and they tackled the scattering problem by a thin layer of high index  in Ref. \cite{artola1991diffraction}.  We also refer the reader  to the work in Ref. \cite{huang2007uniform} in which the authors provide  uniform a priori estimates for static Maxwell interface problems, and  to the works in Ref.  \cite{HJN08,CDP09,CDFP11} for transmission problems with high contrast of electric conductivities.

In this work we investigate a transmission problem  in materials  presenting high contrast of magnetic permeabilities.   
The  domain of interest is made up of two subdomains that represent a magnetic conductor and a non-magnetic material, and the relative magnetic permeability $\mu_{r}$  between the materials can be very high. Our interest is the solvability of  the time-harmonic Maxwell equations  together with uniform energy estimates in $\bL^{2}$ for the electromagnetic field when  $\mu_{r}$ tends to infinity. This issue of uniform estimates is not obvious, and this is especially true since we do not assume that the interface between the subdomains is smooth.

The main result of this paper is uniform a priori estimates for Maxwell solutions when the interface between the two subdomains is supposed to be Lipschitz, cf. Theorem \ref{2T0} in Sec. \ref{S2}. This new result shows in particular that the magnetic field tends to $0$ in energy norm  inside the magnetic conductor   as $\mu_{r}$ tends to infinity. Our result is valid for a right-hand side in $\bL^{2}$. This is notably a difference with the  uniform $\bL^{2}$ estimates  at high conductivity \cite[Theorem 2.1]{CDP09}, which is valid for  data in $\bL^{2}$ with divergence in $\L^{2}$.  

The proof of our main result is based on an appropriate decomposition of the magnetic field into a vector field in the Sobolev space $\bH^{1}$ and a gradient field which is piecewise $\bH^{\frac12}$ (does not belong to the energy space), and the gradient part is estimated thanks to uniform estimates for scalar transmission problems with constant coefficients on two subdomains.  Note that the uniform estimates for these scalar transmission problems (cf. Lemma \ref{2L9} in Sec. \ref{App0}) are a consequence of a more general result \cite[Theorem 2.1]{CDP09}, and also that the decomposition of the magnetic field based on a vector potential technique is rather classical \cite{ABDG98}. However, the proof of the main result on uniform estimates for the electromagnetic field (Theorem \ref{2T0}) is new.  
In particular,  the proof of uniform $\bL^{2}$ estimates for the electric field at high conductivity \cite[Lemma 4.1]{CDP09} does not seem to be easily 
adaptable to our problem   because we consider that the  relative permeability can be unbounded.

As an application of our uniform estimates, we develop also an argument for the convergence of an asymptotic expansion. Assuming {\em smoothness for the interface} between the subdomains, we prove rigorously that the magnetic field  possesses a multiscale expansion in powers of  $\eps=\dfrac1{\sqrt{\mu_{r}}}$ with {\em profile} terms rapidly decaying inside the magnetic conductor,  cf. Sec. \ref{S3bis}.

We also refer the reader to several works which are devoted to asymptotic expansions of the electromagnetic field  at high relative magnetic permeability  (\cite{PePo21,APPK21,AP23}) or at high conductivity (see, {\em e.g.}, \cite{S83,MS84,MS85,HJN08,CDP09,DFP09,CDFP11}) in a domain made up of two subdomains when the interface is smooth: see, {\em e.g.}, \cite{S83,MS84,MS85,PePo21,APPK21,AP23} for plane interface and eddy current approximations, and \cite{HJN08,CDP09,DFP09,CDFP11} for a three-dimensional model of  {\em  skin effect} in non-magnetic materials.

This paper is organized as follows. In Section \ref{S2} we introduce the framework and we state the main result. In Section \ref{S4} we prove  uniform estimates for Maxwell solutions when the relative magnetic permeability is high. In  Section \ref{S3bis} we apply these uniform estimates for proving the convergence of an asymptotic expansion of the magnetic field at high relative magnetic permeability, and we provide the first terms of this expansion. We give also elements of proof for this expansion in Appendix \ref{AppA}.

\section{Framework and main result}
\label{S2}

Let $\Omega$ be a smooth bounded simply connected domain in ${\R}^3$, and $\Omega_{-}\subset\subset\Omega$ be a Lipschitz connected subdomain of $\Omega$. 
We denote by $\Gamma$ the boundary of $\Omega$, and by  $\Sigma$  the boundary of $\Omega_{-}$. 
Finally, we denote by $\Omega_{+}$ the complementary of $\overline{\Omega}_{-}$ in $\Omega$, cf. Figure~\ref{F1}. 

We consider the time-harmonic Maxwell equations given by Faraday's and Amp\`ere's laws in $\Omega$:
\begin{equation}
    \rot \ \EE - i \omega\bm \HH = 0 \quad \mbox{and}\quad
    \rot \ \HH + (i\omega\varepsilon_0 - \bs) \EE =   \jj
    \quad\mbox{in}\quad\Omega\, .
\label{MS}
\end{equation}
Here, $(\EE,\HH)$ represents the electromagnetic field, $\varepsilon_0$ is the electric permittivity, $\omega$ is the angular frequency ($\omega\neq0$ and the time convention is $\exp(-i\omega t)$), $\jj$ represents a current density  and is supposed to belong to  $\bL^2(\Omega)$, $\bs$ is the electric conductivity, and  $\bm$ is the magnetic permeability. We assume that the domain $\Omega$ is made of two connected subdomains $\Omega_+$ and $\Omega_-$ in which the coefficients $\bs$ and $\bm$ take different values $(\sigma_+>0,\sigma_->0)$ and $(\mu_{+}>0, \mu_{-}>0)$, respectively, and we denote by $\mu_{r}$ the relative magnetic permeability between the  subdomains  $\Omega_-$ and $\Omega_+$, and which is defined as: 
\begin{equation}
\label{E0}
 \mu_{r}=\mu_{-}/\mu_{+} \ .
\end{equation}
\begin{figure}[h]
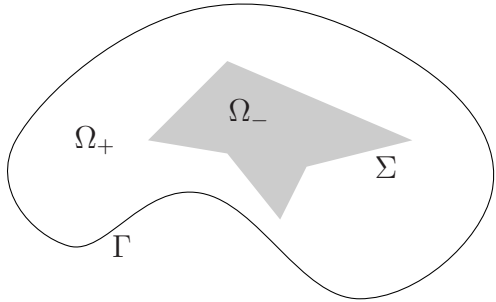

%%%%%%%%%%%%%%%%%%%%%%%%%%%%%%%%
%%%%%%%% DEBUT DESSIN
%%%%%%%%%%%%%%%%%%%%%%%%%%%%%%%
\begin{center}
\figinit{1.0pt}
% pt de la courbe \Partial\Omega
\figpt 1:(-100,20)
\figpt 2:(-30,70)\figpt 3:(50,50)
\figpt 4:(80,0)\figpt 5:(30,-40)
\figpt 6:(-30,0)\figpt 7:(-80,-20)
% pt de la courbe \Sigma
\figpt 8:(-50,20)\figpt 9:(-20,50)\figpt 10:(50,20)
\figpt 12:(10,10)\figpt 13:(0,-10)
\figpt 14:(-20,15)%
% ecriture de \Sigma, \Omega
\figpt 16:(-70,20)\figpt 15:(-10,30)
\figpt 17:(40,10) \figpt 18:(-60,-20)
\figpt 21:(50,65)
% pt vecteurs normaux
\figpt 19:(61,65) \figpt 20:(3,27)
%%%%%%%%%%%%%%%%%%%%%
%% fichier graphique
%%%%%%%%%%%%%%%%%%%
\psbeginfig{}
\pscurve[1,2,3,4,5,6,7,1,2,3]
\pssetfillmode{yes}\pssetgray{0.8}
\psline[8,9,10,12,13,14,8]
\pssetfillmode{no}\pssetgray{0}
%\pscurve[8,9,10,12,13,14,8,9,10]
%Vecteurs
%\psarrow[3,19]
%\psarrow[12,20]
\psendfig
\end{center}
%%%%%%%%%%%%%%%%%%%
%% writing
%%%%%%%%%%%%%%%%%%%
\figvisu{\figBoxA}{}{
\figwritew 15: $\Omega_{-}$(-6pt)
\figwritec [16]{$\Omega_{+}$}
\figwritec [17]{$\Sigma$}
\figwritec [18]{$\Gamma%\partial\Omega
$}
%\figwritew 19: $\nn$(6pt)
%\figwritee 20: $\nn$(3pt)
}
\centerline{\box\figBoxA}
\caption{The domain $\Omega$ and its subdomains $\Omega_-$ and  $\Omega_+$}
\label{F1}
\end{figure}
In this work we are particularly interested in the case where this parameter can be high, which is often the case when the domain $\Omega_-$ represents a linear magnetic conductive material and the domain  $\Omega_+$ represents a non-magnetic material.

To complement the Maxwell harmonic equations \eqref{MS}, we consider either the perfectly insulating electric boundary conditions 
\begin{equation}
\label{PIbc}
   \EE\cdot\nn=0 \quad\mbox{and}\quad \HH\times\nn=0\quad\mbox{on}\quad\Gamma\, ,%.
\end{equation}
or the perfectly conducting electric boundary conditions 
\begin{equation}
\label{Pcbc}
   \EE\times\nn=0 \quad\mbox{and}\quad \HH\cdot\nn=0\quad\mbox{on}\quad\Gamma\,.
\end{equation}

\subsection{Main result: uniform estimates in  Lipschitz domains}
Hereafter, we denote by  $\|\cdot \|_{0,\mathcal{O}}$ the norm in $\bL^2(\mathcal{O})=\L^2(\mathcal{O})^{3}$. In the framework above it is possible to prove uniform a priori estimates for the Maxwell transmission problem:
\begin{thm}
\label{2T0}
Let $\sigma_{\pm}>0$, and  $\mu_{+}>0$. If the interface $\Sigma$ is Lipschitz,  there are constants $\mu_{\star}$ and $C>0$ such that for all $\mu_{r}\geqslant\mu_{\star}$ the Maxwell problem \eqref{MS} with boundary conditions \eqref{PIbc} and data $\jj\in\bL^2(\Omega)$ has a unique solution $(\EE,\HH)$ in $\bL^2(\Omega)^2$, which satisfies:
\begin{equation}
   \|\HH\|_{0,\Omega} + \|\EE\|_{0,\Omega} 
    + \sqrt{\mu_{r}}\, \|\HH\|_{0,\Omega_-}    \leqslant C \|\jj\|_{0,\Omega} . 
\label{3E1}
\end{equation}
A similar result holds for boundary conditions \eqref{Pcbc}. 
\end{thm}
This result is proved in Section \ref{S4}. The proof  is based on an appropriate decomposition of the magnetic field into a regular field in $\bH^{1}(\Omega)$ and a gradient field. The gradient part  is estimated thanks to piecewise $\mathrm{H}^{3/2}$ uniform estimates for a scalar transmission problem with constant coefficients on two subdomains (cf.  Lemma \ref{2L9} in Section \ref{App0}). 

In that follows, we introduce for convenience the following parameter in the Maxwell equations \eqref{MS} 
\begin{equation}
\label{epsilon} 
\eps=\dfrac1{\sqrt{\mu_{r}}} 
\,.
\end{equation}
Hence, $\eps$ tends to $0$ as  $\mu_{r}$ tends to infinity. Then  for $\mu_{r}\geqslant\mu_{\star}$, where the constant $\mu_{\star}$ is given by Theorem \ref{2T0}, 
 we denote by $(\EE_{(\eps)},\HH_{(\eps)})$ the solution of the system \eqref{MS} --\eqref{PIbc}.  As an application of uniform estimates \eqref{3E1} we will develop an argument for the convergence of  an asymptotic expansion of $(\EE_{(\eps)},\HH_{(\eps)})$  as $\eps\to0$,  see Section \ref{Svalid}. 
\Bk

\subsection{Magnetic formulation}
We deduce from the Maxwell system \eqref{MS}-\eqref{PIbc} the following variational formulation for the magnetic field $\HH_{(\eps)}$. The variational space is the Hilbert space $\bH_0(\rot,\Omega)$:
\begin{equation}
 \bH_0(\rot,\Omega) =\{\uu\in\bL^2(\Omega)\; |\ \rot\uu\in\bL^2(\Omega),\, 
 \uu\times\nn = 0  \ \mbox{on}\ \Gamma\}\ , 
\end{equation}
endowed with its graph norm, and the variational problem writes
\\[1ex]
{\it Find $\HH_{(\eps)} \in \bH_0(\rot,\Omega)$ such that for all $\KK \in \bH_0(\rot,\Omega)$}
\begin{multline}
\int_{\Omega} \left(\left(1+ i \frac{\bs}{\omega\varepsilon_{0}}\right)^{-1}\rot\HH_{(\eps)} \cdot \rot\overline\KK- \kappa_{+}^2\mur(\eps)\HH_{(\eps)}\cdot\overline\KK \right) \dr\xx = 
\\
  \int_{\Omega}   \left( 1+ i \frac{\bs}{\omega\varepsilon_{0}}
 \right)^{-1}\jj\cdot\rot\overline\KK\,\dr\xx\ .
\label{FVH0}
\end{multline}
Here we have set
\begin{equation}
\label{2Eeps}
  \mur(\eps)={\mathbf{1}_{\Omega\iso}}  
   + {\frac{1}{\eps^2}}\,{\mathbf{1}_{\Omega\con}} 
 \ ,    \quad 
 \quad\mbox{and}\quad    \kappa_{+}:=\omega\sqrt{\varepsilon_{0}\mu_{+}} \ .
\end{equation}

\section{Proof of uniform estimates for Maxwell solutions at high relative magnetic permeability}
\label{S4}

We consider the harmonic Maxwell system \eqref{MS}-\eqref{PIbc}. We are going to prove the following statement:
\begin{lem}
\label{5L1}
Let $\sigma_{\pm}>0$, and  $\mu_{+}>0$. There are constants $\mu_\star$ and $C_0>0$ such that if $\mu_{r}\geqslant \mu_\star$ any solution $(\EE,\HH)\in\bL^2(\Omega)^2$ of problem \eqref{MS} with boundary condition \eqref{PIbc} and data $\jj\in\bL^2(\Omega)$ 
satisfies the estimate
\begin{equation}
  \|\HH\|_{0,\Omega} \leqslant C_0 \|\jj\|_{0,\Omega}.
\label{5E1}
\end{equation}
A similar statement holds for boundary conditions \eqref{Pcbc}.
\end{lem}

This lemma is the key for the proof of Theorem \ref{2T0} and is going to be proved in the next subsection.  As a consequence of this lemma, we will obtain estimates \eqref{3E1}:
\begin{cor}
\label{5C1}
Let  $\sigma_{\pm}>0$, and  $\mu_{\pm}>0$. Let $(\EE,\HH)\in\bL^2(\Omega)^2$ be solution of problem \eqref{MS} with boundary condition \eqref{PIbc} and data $\jj\in\bL^2(\Omega)$. Then we have $\Div(\bm \HH)=0$ in $\Omega$,
and if $\HH$ satisfies estimate \eqref{5E1}, then setting  $\mathrm{m}=\min(\sqrt{\omega^2 \varepsilon_0^{2} +\sigma^{2}_{-}}, \sqrt{\omega^2 \varepsilon_0^{2} +\sigma^{2}_{+}})$,
\[
C_{1}=   \frac{\max\left( \frac{\omega^2 \varepsilon_0^{2} +\sigma^{2}_{-}}{ \sigma_{-}},  \frac{\omega^2 \varepsilon_0^{2} +\sigma^{2}_{+}}{ \sigma_{+}} \right)}{\mathrm{m}}
  \, , \quad\mbox{and} \quad
 C_{2}= \sqrt{ \frac{\varepsilon_0 C^{2}_{1}}{\mathrm{m}^{2}
 }+ \frac{C_{1}}{\omega\,\mathrm{m}
 }}
\,, 
\]
there holds 
\begin{equation}
   \|\HH\|_{0,\Omega} + \|\rot\HH\|_{0,\Omega} 
   +\sqrt{\mu_{-}} \, \|\HH\|_{0,\Omega_-} \leqslant (C_{0} +C_1 +C_{2}) \|\jj\|_{0,\Omega
   }\ . 
\label{5E2}
\end{equation}
A similar estimate holds for boundary conditions \eqref{Pcbc}.
\end{cor}

This result is going to be proved in Section \ref{4.2}.  Finally, estimate \eqref{5E1} implies existence and uniqueness of solutions.

\begin{cor}
\label{5C2}
Let $\sigma_{\pm}>0$, and  $\mu_{\pm}>0$. We assume that estimate \eqref{5E1} holds for any solution $(\EE,\HH)\in\bL^2(\Omega)^2$ of problem  \eqref{MS}-\eqref{PIbc} with $\jj\in\bL^2(\Omega)$. 
Then for any $\jj\in\bL^2(\Omega)$, there exists a unique solution $(\EE,\HH)\in\bL^2(\Omega)^2$ of problem \eqref{MS}-\eqref{PIbc}.
A similar result holds for boundary conditions \eqref{Pcbc}.
\end{cor}

This result is going to be proved in Section \ref{4.3}. Finally we deduce Theorem \ref{2T0} from the previous three statements.

\subsection{Preliminary lemma}
\label{App0}

\newcommand  {\ba}{\underline a}

In this section we provide a uniform piecewise estimate for the solution of a scalar transmission problem with constant coefficients on two subdomains. We will use this result in the proof of Lemma \ref{5L1} (cf. Section \ref{Sec3.1}) for obtaining a uniform $\bL^2$ estimate for the magnetic field. 
 
 For any given function  $\ba=(a_+,a_-)$ determined by the two constants $a_\pm$ on $\Omega_\pm$, we consider the following variational problem : Find $\varphi \in\H_{0}^1(\Omega)$ such that 
\begin{multline}
   \forall \psi \in  \H_{0}^1(\Omega), \quad 
   \int_{\Omega_{+}} a_+\nabla\varphi^+ \cdot \nabla\overline{\psi}{}^+ \,\dr\xx
   + \int_{\Omega_{-}} a_-\nabla\varphi^- \cdot \nabla\overline{\psi}{}^-\, \dr\xx= \\[-1ex]
     (a_-  - a_+) \!\int_{\Sigma} g\; \overline{\psi} \, \dr s\, ,
\label{1E1}
\end{multline}
where the right-hand side $g$ satisfies the regularity assumption
\begin{equation}
\label{1E7}
g\in \L^2(\Sigma)
\end{equation}
and the extra compatibility condition
\begin{equation}
 \int_{\Sigma} g \, \dr s =0 
 \,.
\label{1E2}
\end{equation}

In the framework of Section \ref{S2} (we recall that the surface $\Sigma$ is Lipschitz), we have the following uniform piecewise estimate for the solution of the problem \eqref{1E1}.
 \begin{lem}
\label{2L9}
Let us assume that $a_+\neq0$.
There exist a constant $\rho_{0}>0$ independent of $a_+$ such that for all $a_-\in\{z\in\C \, | \,|z|\geqslant \rho_{0}|a_+| \}$, the problem \eqref{1E1} with a right-hand side  $g$ satisfying \eqref{1E7}-\eqref{1E2}  has a unique solution $\varphi\in \H_{0}^1(\Omega)$ which moreover is piecewise $\H^{3/2}$ and satisfies the uniform estimate
\begin{equation}
\label{ue}
   \|\varphi^+\|_{\frac32,\Omega_+} + \|\varphi^-\|_{\frac32,\Omega_-}
   \leqslant C_{\rho_{0}} \|g\|_{0,\Sigma}\ ,
   \end{equation}
with a constant $C_{\rho_{0}}>0$, independent of $a_+$, $a_-$, and $g$.
\end{lem}
This result is obtained as a consequence of   \cite[Theorem 2.1]{CDP09}.

\begin{rem}
The assumption on the interface  $\Sigma$, that is a Lipschitz surface, is necessary  in Lemma \ref{2L9}. Indeed, there exists non-Lipschitz surfaces such that estimate \eqref{ue} does not hold, for instance in the case of a checkerboard configuration in a two dimensional domain, cf. \cite[Theorem 8.1]{C-D-N99} and \cite[Rem 2.3]{CDP09}. 
\end{rem}

\subsection{Proof of Lemma \ref{5L1}\,: Uniform $\bL^2$ estimate of the magnetic field}
\label{Sec3.1}
Reductio ad absurdum: We assume that there is a sequence $(\EE_m,\HH_m)\in\bL^2(\Omega)^2$, $m\in\N$, of solutions of the Maxwell system \eqref{MS}-\eqref{PIbc} associated with a magnetic  permeability  $\bm_{m}=(\mu_{+},\mu_m \, \mu_{+})$ and a right hand side $\jj_m\in%\bH_{0}(\Div,
\bL^2(\Omega)$:
\begin{subequations}
\begin{gather}
\label{5E3a}
    \rot  \EE_m - i \omega\bm_{m} \HH_m = 0 \quad\mbox{in}\quad\Omega\,,\\
\label{5E3b}
    \rot \HH_m + (i\omega\varepsilon_0 - \bs) \EE_m = \jj_m \quad\mbox{in}\quad\Omega\,,\\
\label{5E3c}
  \EE_m\cdot\nn=0  \quad\mbox{and}\quad \HH_m\times\nn = 0 \quad\mbox{on}\quad\partial\Omega\,,
\end{gather}
\end{subequations}
satisfying the following conditions
\begin{subequations}
\begin{eqnarray}
\label{5E4a}
   &\mu_m\to\infty\quad &\mbox{as \ $m\to\infty$,} \\
\label{5E4b}
   &     \|\HH_{m}\|_{0,\Omega}    = 1\quad &\mbox{$\forall m\in\N$,} \\
\label{5E4c}
   &\|\jj_m\|_{0,\Omega}\to0\quad&\mbox{as \ $m\to\infty$.}
\end{eqnarray}
\end{subequations}

First, we particularize the magnetic variational formulation \eqref{FVH0} for the sequence $\{\HH_m\}$: For all $\HH'\in\bH_{0}(\rot,\Omega)$:
\begin{equation}
\label{5EVm}
   \int_{\Omega}\left( \frac1{i\omega \varepsilon_0 -\bs} \rot\HH_m \cdot \rot\overline{\HH'}  
   +i \omega \bm_{m}  \HH_m \cdot \overline{\HH'} \right)\;\dr\xx 
    =   \int_{\Omega}   
  \frac{\jj_{m}}{i\omega \varepsilon_0 -\bs} 
 \cdot\rot\overline\HH'\,\dr\xx
\,.
\end{equation}
Choosing $\HH'=\HH_m$ in \eqref{5EVm} and taking the real part, and then using the Cauchy-Schwarz inequality, we obtain the following uniform bound on the curl of the magnetic field in $\Omega$: 
\begin{equation}
\label{5ECH}
  \| \rot\HH_m \|_{0,\Omega}  \leqslant   \frac{\max\left( \frac{\omega^2 \varepsilon_0^{2} +\sigma^{2}_{-}}{ \sigma_{-}},  \frac{\omega^2 \varepsilon_0^{2} +\sigma^{2}_{+}}{ \sigma_{+}} \right)}
 {\min(\sqrt{\omega^2 \varepsilon_0^{2} +\sigma^{2}_{-}}, \sqrt{\omega^2 \varepsilon_0^{2} +\sigma^{2}_{+}})}
       \| \jj_m \|_{0,\Omega}  \, .
 \end{equation}

\subsubsection{Decomposition of the magnetic field and bound in $\H^{\frac12}$}
We recall that we have assumed that the domain $\Omega$ is simply connected and has a smooth connected boundary. Relying to  Theorem 2.12 and Theorem 3.17 in \cite{ABDG98}, for all $m\in\N$ there exists a unique $\vv_m\in\bH^1(\Omega)$ such that 
\begin{equation}
\label{5E5}
   \rot\vv_m=\rot\HH_m,\quad  \Div \vv_m=0  \ \mbox{ in }\ \Omega \ , 
      \quad\mbox{and}\quad
   \vv_m\times\nn=0\ \mbox{ on }\ \partial\Omega  \,.
\end{equation}
Moreover, we have the estimate
\begin{equation}
\label{5E6}
   \|\vv_m\|_{1,\Omega} \leqslant C \| \rot\HH_m \|_{0,\Omega}\,,
\end{equation}
where $C$ is independent of $m$. As a consequence of the equality $\rot\vv_m=\rot\HH_m$ and the simple connectedness of $\Omega$, we obtain that there exists $\varphi_m\in\H_{0}^1(\Omega)$ such that 
\begin{equation}
\label{5E7H}
   \HH_m = \vv_m + \nabla \varphi_m\,.
\end{equation}
Then using \eqref{5E7H} we write equation \eqref{5E3a} as
\[
  \rot  \EE_m - i \omega\bm_{m} ( \vv_m + \nabla \varphi_m ) = 0 \quad\mbox{in}\quad\Omega .
\]
Let $\psi\in\H_{0}^1(\Omega)$ be a test function. Multiplying the above equality by $\nabla \overline\psi$ and integrating over $\Omega$, we obtain, using that $\Div\vv_m=0$:
\begin{equation}
\label{5E8}
   \int_{\Omega}   \bm_{m}\;\nabla\varphi_m \cdot  \nabla\overline{\psi}\;  \dr\xx = 
  (\mu_{m}- 1) \mu_{+} \int_{\Sigma}  \vv_m \cdot \nn\on \Sigma  \; \overline\psi \; \dr s\, ,
\end{equation}
where $\nn$ denotes the unit normal vector on $\Sigma$, inwardly oriented to $\Omega_{-}$. Note that the boundary value $\psi=0$  on $\partial\Omega$ has been used in \eqref{5E8}.

Thus $\varphi_m$ is solution of the Dirichlet problem defined by the variational equation \eqref{5E8}. Since 
\[
   \vv_m\cdot\nn\in\L^2(\Sigma) \quad\mbox{and}\quad
   \int_{\Sigma}  \vv_m\cdot\nn\on \Sigma \;\dr s = \int_{\Omega_-} \Div\vv_m \,\dr\xx = 0,
\]
the Dirichlet problem defined by \eqref{5E8} satisfies the assumptions of Lemma \ref{2L9} (cf. Section \ref{App0}) with $a_-=\mu_{m} \mu_{+} $ and $a_+ = \mu_{+}$, and $g=\vv_m\cdot\nn\on \Sigma$. 
Therefore we obtain the following uniform estimate for $\mu_m$ large enough (i.e.\ for $m$ large enough, cf.\ \eqref{5E4a})
\begin{equation*}
   \|\varphi^+_m\|_{\frac32,\Omega_{+}} + \|\varphi^-_m\|_{\frac32,\Omega_{-}}\leqslant 
   C_{0}
    \|\vv_m \cdot \nn\|_{0,\Sigma}.
\end{equation*}
Since $\|\vv_m \cdot \nn\|_{0,\Sigma}$ is bounded by $\|\vv_m\|_{1,\Omega}$, the last inequality implies 
\begin{equation}
\label{5E9bH}
   \|\varphi^+_m\|_{\frac32,\Omega_{+}} + \|\varphi^-_m\|_{\frac32,\Omega_{-}}\leqslant 
   C_{0} 
   \|\vv_m\|_{1,\Omega} . 
\end{equation}
Finally \eqref{5ECH}, 
  \eqref{5E6}, and \eqref{5E9bH} implies that
\begin{equation}
\label{5E11H}
   \|\varphi^+_m\|_{\frac32,\Omega_{+}} + \|\varphi^-_m\|_{\frac32,\Omega_{-}}
   + \|\vv_m\|_{1,\Omega}
    \leqslant B
\end{equation}
for a constant $B>0$ independent of $m$. With \eqref{5E7H}, \eqref{5E11H} gives that the sequence $\{\HH_m\}$ is bounded in $\H^{\frac12}$ on $\Omega_-$ and $\Omega_+$:
\begin{equation*}
   \|\HH^+_m\|_{\frac12,\Omega_{+}} + \|\HH^-_m\|_{\frac12,\Omega_{-}}
   \leqslant B .
\end{equation*}
Combining the above bound with \eqref{5ECH}, 
we obtain the uniform bound
\begin{equation}
\label{5E12H}
   \|\HH^+_m\|_{\frac12,\Omega_{+}} + \|\HH^-_m\|_{\frac12,\Omega_{-}} 
+      \| \rot\HH_m \|_{0,\Omega} 
   \leqslant C .
\end{equation}

\subsubsection{Limit of the sequence and conclusion}
The domains $\Omega_\pm$ being bounded, the embedding of $\bH^{\frac12}(\Omega_\pm)$ in $\bL^2(\Omega_\pm)$ is compact. Hence as a consequence of \eqref{5E12H}, we can extract a subsequence  of $\{\HH_m\}$ (still denoted by $\{\HH_m\}$) which is converging in $\bL^2(\Omega)$. By the Banach-Alaoglu theorem, we can assume that the sequence $\{\rot\HH_m\}$ is weakly converging in $\bL^2(\Omega)$: We deduce that there is $\HH\in \bL^2(\Omega)$ such that
\begin{equation}
\left\{
   \begin{array}{lll}
   \label{4E6}
   \rot\HH_m \rightharpoonup \rot\HH   \quad &\mbox{in} \quad \bL^2(\Omega)
\\
   \HH_m \rightarrow \HH   \quad &\mbox{in} \quad \bL^2(\Omega).
   \end{array}
\right.
\end{equation}
A consequence of the strong convergence   of $\{\HH_m\}$ in $\bL^2(\Omega)$ and \eqref{5E4b} is that 
\begin{equation}
\label{4E7}
 \|\HH\|_{0,\Omega}=1\, .
 \end{equation}
 We are going  to prove that 
\begin{equation}
\label{4E7a}
\HH=0 \quad \mbox{in} \quad \Omega \ ,
\end{equation}
which will contradict \eqref{4E7}, and finally prove estimate \eqref{5E1}.

 Choosing $\HH'=\HH_m$ in  the magnetic variational formulation  \eqref{5EVm}  for the sequence $\{\HH_m\}$, then taking the imaginary part and  using the Cauchy-Schwarz inequality, we obtain
\begin{multline}
\label{5EVHim}
\omega  \mu_{m}\mu_{+} \|\HH_{m}^{-}\|^{2}_{0,\Omega_{-}}  +  \omega\mu_{+} \|\HH_{m}^{+}\|^{2}_{0,\Omega_{+}} \leqslant 
\\
\frac{\omega \varepsilon_0}{\omega^{2} \varepsilon^{2}_0+\sigma_{-}^{2}}  \| \rot\HH^{-}_m \|^{2}_{0,\Omega_{-}}  
+\frac{\omega \varepsilon_0}{\omega^{2} \varepsilon^{2}_0+\sigma_{+}^{2}}  \| \rot\HH^{+}_m \|^{2}_{0,\Omega_{+}}  
+\| \frac{\jj_m}{i\omega\varepsilon_0-\bs}\|_{0,\Omega}   \|\rot\HH_{m}\|_{0,\Omega} 
 \,.
\end{multline}
Hence, according to 
\eqref{5ECH}-\eqref{5E4c} 
  we obtain that the right-hand side in \eqref{5EVHim} tends to zero and we infer  the convergence results  
\begin{equation}
\label{4E7b}
 \sqrt{\mu_{m}} \|\HH_{m}^{-}\|_{0,\Omega_{-}}    \to 0 \quad \mbox{as \ $m\to\infty$,}  \quad \mbox{and}  \quad \|\HH_{m}^{+}\|_{0,\Omega_{+}} \to 0 \quad \mbox{as \ $m\to\infty$.}
\end{equation}
Finally, according to \eqref{5E4a}-\eqref{4E6}, we deduce that  $\HH^-:=\HH\on{\Omega_-}$ and $\HH^+:=\HH\on{\Omega_+}$ satisfy 
\begin{equation}
\label{4E7b}
\HH^- = 0 \quad \mbox{in \ $\Omega_{-}$,}  \quad \mbox{and}  \quad \HH^+ = 0  \quad  \mbox{in \ $\Omega_{+}$,} 
\end{equation}
which proves \eqref{4E7a}, contradicts \eqref{4E7} and finally ends the proof of Lemma \ref{5L1}.

\subsection{Proof of Corollary \ref{5C1}}
\label{4.2}

Let $(\EE,\HH)\in\bL^2(\Omega)^2$ be a solution of the Maxwell problem \eqref{MS} with boundary condition \eqref{PIbc} and data $\jj\in\bL^{2}(\Omega)$
,  and we assume that 
\begin{equation}
   \|\HH\|_{0,\Omega} \leqslant C_0 \|\jj\|_{0,\Omega}.
\label{5E21}
\end{equation}
First, taking the divergence of equation $\rot\EE - i\omega\bm \HH = 0$ in $\Omega$, we immediately obtain
\begin{equation}
\label{4E14}
   \Div(\bm\HH) = 0\quad \mbox{in} \quad \Omega . 
  \end{equation}
Second, $\HH\in\bH_{0}(\rot,\Omega)$ is solution of the variational problem \eqref{FVH0}. Taking as test function $\HH$ itself, we obtain the identity 
\begin{equation}
\label{5E22}
   \int_{\Omega}\left( \frac1{i\omega \varepsilon_0 -\bs} \rot\HH \cdot \rot\overline{\HH}  
   +i \omega \bm  \HH \cdot \overline{\HH} \right)\;\dr\xx 
   =  \int_{\Omega} \frac{\jj}{i\omega \varepsilon_0 -\bs}\cdot \rot\overline{\HH}\,\dr\xx\,. 
\end{equation}
Taking the real part of \eqref{5E22}, we obtain
\begin{equation*}
\frac{\sigma_{-}}{\omega^2 \varepsilon_0^{2} +\sigma^{2}_{-}}    \|\rot\HH\|^2_{0,\Omega_{-}} +   \frac{\sigma_{+}}{\omega^2 \varepsilon_0^{2} +\sigma^{2}_{+}}   \|\rot\HH\|^2_{0,\Omega_{+}} =
\Re\,( \frac{\jj}{\bs-i\omega \varepsilon_0} , \rot \HH )_{0,\Omega} 
\end{equation*}
hence, using Cauchy-Schwarz inequality, and setting $C_{1}=   \frac{\max\left( \frac{\omega^2 \varepsilon_0^{2} +\sigma^{2}_{-}}{ \sigma_{-}},  \frac{\omega^2 \varepsilon_0^{2} +\sigma^{2}_{+}}{ \sigma_{+}} \right)}
 {\min(\sqrt{\omega^2 \varepsilon_0^{2} +\sigma^{2}_{-}}, \sqrt{\omega^2 \varepsilon_0^{2} +\sigma^{2}_{+}})}$, we obtain 
\begin{equation}
   \|\rot\HH\|_{0,\Omega}\leqslant C_{1}   \| \jj\|_{0,\Omega}.
\label{4E12}
\end{equation}
Then, taking the imaginary part of \eqref{5E22},  since $\omega\neq 0$, we obtain 
\begin{multline*}
  \mu_{-}\|\HH\|^2_{0,\Omega_-} +   \mu_{+}\|\HH\|^2_{0,\Omega_+} 
  = \frac{\varepsilon_0}{\omega^2 \varepsilon_0^{2} +\sigma^{2}_{-}}    \|\rot\HH\|^2_{0,\Omega_{-}} +   \frac{\varepsilon_0}{\omega^2 \varepsilon_0^{2} +\sigma^{2}_{+}}   \|\rot\HH\|^2_{0,\Omega_{+}}
  \\ +\frac1\omega  \Im\,( \frac{\jj}{i\omega \varepsilon_0-\bs} , \rot \HH )_{0,\Omega} ,
\end{multline*}
 hence, using Cauchy-Schwarz inequality and  inequality \eqref{4E12}, we infer
\begin{equation}
\label{4E13}
   \sqrt{\mu_{-}}\, \|\HH\|_{0,\Omega_-} \leqslant    
    \sqrt{ \frac{\varepsilon_0 C^{2}_{1}}{\min(\omega^2 \varepsilon_0^{2} +\sigma^{2}_{-}, \omega^2 \varepsilon_0^{2} +\sigma^{2}_{+})}+ \frac{C_{1}}{\omega\,\min(\sqrt{\omega^2 \varepsilon_0^{2} +\sigma^{2}_{-}}, \sqrt{\omega^2 \varepsilon_0^{2} +\sigma^{2}_{+}})}}
     \,\| \jj\|_{0,\Omega}.
\end{equation}
Formulae \eqref{5E21}-\eqref{4E12}-\eqref{4E13} yield Corollary \ref{5C1}.

\subsection{Proof of Corollary \ref{5C2}}
\label{4.3}

Let $\mu_{\pm}$, $\sigma_{\pm}$ and $\omega$ (i.e., $\kappa_{+}$) be fixed. Let us define $\mur$ the piecewise constant function  on $\Omega$ as 
(we recall $\mu_{r}=\mu_{-}/\mu_{+}$): 
\begin{equation}
\label{5E31}
   \mur = \mathbf{1}_{{\Omega_+}}  +  
    \mu_{r} \,\mathbf{1}_{\Omega_-}.
\end{equation}
Then, according to \eqref{epsilon}-\eqref{2Eeps}  the sesquilinear form in the left hand side of \eqref{FVH0} writes 
\begin{equation}
\label{5E32}
   \int_{\Omega} \left( (1+ i \frac{\bs}{\omega\varepsilon_{0}}
)^{-1}\rot\HH \cdot \rot\overline{\HH'}  
   - \kappa_{+}^2 \,\mur\, \HH \cdot \overline{\HH'} \right)\;\dr\xx .
\end{equation} 
The proof of Corollary \ref{5C2} relies on a regularization procedure:
We consider the functional space
$$
   \XX_{\NN}(\Omega,\mur) = \{\HH \in \bH_{0}(\rot,\Omega) | \quad 
   \Div(\mur\HH) \in \L^2(\Omega)
   \}.
$$
Let $s>0$ be a real number which will be fixed later. Let us introduce the sesquilinear forms $A_s$ and $B$\,: $\XX_{\NN}(\Omega,\mur) \times\XX_{\NN}(\Omega,\mur) \rightarrow\C$
\begin{subequations}
\begin{gather}
   A_s(\HH,\HH') =
   \int_{\Omega}  \left(   \left( 1+ i \frac{\bs}{\omega\varepsilon_{0}}
 \right)^{-1}\rot\HH \cdot \rot\overline{\HH'}  
   + s\, \Div\mur\HH\, \Div\overline{\mur\HH'} \right) \,\dr\xx \\
   B(\HH,\HH') =
   \int_{\Omega}   
   \mur \HH \cdot \overline{\HH^\prime}\,\dr\xx\,.
\end{gather}
\end{subequations}
The regularized variational formulation is: Find $\HH \in  \XX_{\NN}(\Omega,\mur)$ such that
\begin{equation}
\label{5E35}
   \forall\,\HH' \in  \XX_{\NN}(\Omega,\mur) ,\quad
   A_s(\HH,\HH') - \kappa_{+}^2 B(\HH,\HH') = 
   \int_{\Omega}   \left( 1+ i \frac{\bs}{\omega\varepsilon_{0}}
 \right)^{-1}\jj\cdot\rot\overline\HH'\,\dr\xx
\end{equation}
As a consequence of \cite[Th.\,7.1]{CoDa00}, we obtain that if
\begin{equation}
\label{5E36}
   \frac{\kappa_{+}^2}{s} \ \ \mbox{is not an eigenvalue of the Dirichlet problem for the operator}
   \ \ - \Div\mur\nabla,
\end{equation}
then any solution $(\EE,\HH)\in\bL^2(\Omega)^2$ of problem \eqref{MS}-\eqref{PIbc} with $\jj\in\bL^{2}(\Omega)$  
provides a solution of problem \eqref{5E35}, and conversely, any solution $\HH$ of \eqref{5E35} provides a solution of \eqref{MS}-\eqref{PIbc} by setting $\EE
= (i\omega\varepsilon_0 - \bs)^{-1}  (\jj - \rot \ \HH )$.

Thus, we fix $s$ such that \eqref{5E36} holds.

Since the form $A_s$ is coercive on $\XX_{\NN}(\Omega,\mur)$ and the embedding of $\XX_{\NN}(\Omega,\mur)$ in $\bL^2(\Omega)$  is compact, we obtain that the Fredholm alternative is valid: If the kernel of the adjoint problem to \eqref{5E35}
\begin{equation}
\label{5E35k}
   \mbox{Find }\ \HH' \in  \XX_{\NN}(\Omega,\mur),\quad 
   \forall\,\HH \in  \XX_{\NN}(\Omega,\mur),\quad
   A_s(\HH,\HH') - \kappa_{+}^2 B(\HH,\HH') = 0,
\end{equation}
is reduced to $\{0\}$, then problem \eqref{5E35} is solvable.

We see that the assumption of Corollary \ref{5C2} implies that \eqref{5E35k} has only the zero solution, and that the same holds for the direct problem
\begin{equation}
   \mbox{Find }\ \HH \in  \XX_{\NN}(\Omega,\mur),\quad 
   \forall\,\HH' \in  \XX_{\NN}(\Omega,\mur),\quad
   A_s(\HH,\HH') - \kappa_{+}^2 B(\HH,\HH') = 0.
\end{equation}

All this implies the unique solvability of problem \eqref{MS}-\eqref{Pcbc} with $\jj\in\bL^{2}(\Omega)$.

\section{Multiscale expansion of the magnetic field in smooth domains}
\label{S3bis}

This section is concerned with a rigorous multiscale expansion for the magnetic field at high relative magnetic permeability. In Section \ref{S3.1bis} we provide the first terms of this expansion.  Then  in  Section \ref{Svalid} we prove error estimates for this expansion. 

For the sake of completeness we give elements of proof for the multiscale expansion in Appendix \ref{AppA}.  We expand the Maxwell operator in power series of $\eps$, cf.  \S\ref{AH1}. Then we deduce the equations satisfied by  the coefficients of the magnetic field, cf. \S\ref{AH2}. We derive explicitly the first terms of this expansion in \S\ref{AH3}. \Bk

In that follows we assume that right-hand side $\jj$ in \eqref{MS} is smooth and that Assumption \ref{H2} holds:
\begin{hyp}
\label{H2}
We assume that the surfaces $\Sigma$ (interface) and $\Gamma$ (external boundary) are smooth.
\end{hyp}

Let $\sigma_{\pm}>0$, $\mu_{+}>0$, and recall $\eps=\dfrac1{\sqrt{\mu_{r}}}>0$ (cf. \eqref{epsilon}). By Theorem \ref{2T0} there exists  $\eps_{\star}>0$ such that for all $\eps\in(0,\eps_{\star})$, the Maxwell problem \eqref{MS} with boundary conditions \eqref{PIbc}  has a unique solution $(\EE_{\eps},\HH_{\eps})$.  In that follows, the magnetic field  $\HH_{(\eps)}$ is denoted by $\HH^{+}_{(\eps)}$ in the non-magnetic part $\Omega_{+}$, and by $\HH^{-}_{(\eps)}$ in the magnetic conducting part $\Omega_{-}$. Then both parts possess series expansions in powers of $\eps$: 
\begin{gather}
\label{6E4abis}
   \HH^{+}_{(\eps)}(\xx) \approx\sum_{j\geqslant0} \eps^j\HH^{+}_j (\xx)  
     \, ,\quad \xx\in\Omega_+\ ,
\\
\label{6E4bbis}
 \HH^{-}_{(\eps)}(\xx) \approx\sum_{j\geqslant0} \eps^j\HH^{-}_j (\xx; \eps)  ,\quad\xx\in\Omega_- \ ,  \quad\mbox{with}\quad   
  \HH^-_j(\xx;\eps) = \chi(y_3) \,\VV_j(y_\alpha,\frac{y_3}{\eps})\, . 
\end{gather}
In \eqref{6E4bbis},  $(y_{\alpha},y_{3})$ is a {\em normal coordinate system} to the surface $\Sigma$ which is defined in  $\cU_-$,  
a tubular neighborhood of the surface $\Sigma$ in the domain $\Omega_{-}$ (cf. Figure~\ref{Fig1}): $y_\alpha$ ($\alpha=1,2$) are tangential coordinates on $\Sigma$, and $y_3$ is the normal coordinate to $\Sigma$, cf. {\it e.g.}, \cite{CDFP11}. The function $\yy\mapsto\chi(y_3)$ is a smooth cut-off with support in $\overline\cU_-$ and equal to $1$ in a smaller tubular neighborhood of $\Sigma$. The vector fields 
$\VV_{j}: (y_\alpha,Y_3)\mapsto\VV_{j}(y_\alpha,Y_3)$ are \textit{profiles} defined on $\Sigma\times\R^+$: They are exponentially decreasing with respect to $Y_{3}$ and are smooth in all variables. 

\input contribF4T.tex
\begin{figure}[ht]
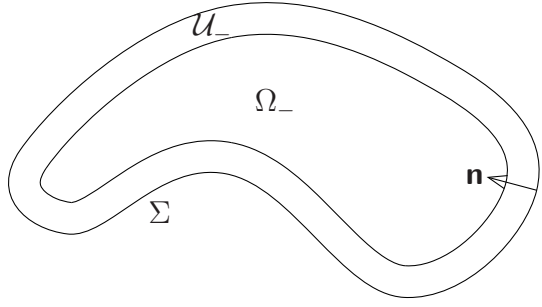

\begin{center}
%%%%%%%%%%%%%%%%%%%%%%%%%%%%%%%%
%%%%%%%% DEBUT DESSIN
%%%%%%%%%%%%%%%%%%%%%%%%%%%%%%%
\def\EpsTube{-10}
%\begin{center}
\figinit{1.2pt}
\pssetupdate{yes}
% pt de la courbe \Sigma
\figpt 8:(-48,0)
% pt de la courbe \Sigma_{h}
\figpt 22:(-92,15.5)\figpt 23:(-27.2,60.8)
\figpt 24:(46.3,41.5)\figpt 25:(70.8,2.8)
\figpt 26:(29,-31)\figpt 27:(-27.2,8)\figpt 28:(-77,-11)
% pt de la courbe \Sigma
\figpt 2:(-24,51)
\figpt 3:(50,50)
% ecriture de \Sigma, \Omega
\figpt 16:(-70,20)\figpt 15:(0,30)
\figpt 17:(40,10) \figpt 18:(-60,-20)
\figpt 21:(50,65) \figpt 31:(30,-18)  \figpt 34:(-30,54) \figpt 36:(-200,30)
 \figpt 35:(-80,18) 
% pt vecteurs normaux
\figpt 19:(61,65) \figpt 20:(3,27)  \figpt 30:(42.6,33)  \figpt 32:(55,7) \figpt 33:(61.6,6.1) 
%%%%%%%%%%%%%%%%%%%%%
%% fichier graphique
%%%%%%%%%%%%%%%%%%%
\psbeginfig{}
\pssetfillmode{no}\pssetgray{0}
\pscurve[22,23,24,25,26,27,28,22,23,24]
\figptscontrolcurve 40,\NbC[22,23,24,25,26,27,28,22,23,24]
\psEpsLayer \EpsTube,\NbC[40,41,42,43,44,45,46,47,48,49,50,51,52,53,54,55,56,57,58,59,60,61]
%\'epaisseur h
%\psline[23,2]
%Vecteurs
\psarrow[25,32]
\pssetfillmode{yes}\pssetgray{0.5}
\psendfig
%\end{center}
%%%%%%%%%%%%%%%%%%%
%% writing
%%%%%%%%%%%%%%%%%%%
\figvisu{\figBoxA}{}
{
\figwritew 15: $\Omega\con$(6pt)
\figwritec [34]{$\cU_{-}$
%$h$
}
\figwrites 8: $\Sigma$ (1pt)
%\figwriten 31: $\Sigma_{h}$ (2pt)
\figwritew 32: $\nn$(1pt)
%\figwritec [35]{$\cU_{-}$}
%\figwriten 35: $\cU_{-}$ (1pt)
\figsetmark{$\figBullet$}
}
\centerline{\box\figBoxA}
 \caption{A tubular neighbourhood of the surface $\Sigma$}
\label{Fig1}	
\end{center}
\end{figure}

\subsection{First terms of the multiscale expansion}
\label{S3.1bis} 

In this section, we provide the construction of the first profiles $\VV_{j}=(\cV_{j},\sv_{j})$ and of the first terms $\HH^+_j$. The first profile $\VV_0$  in the magnetic conductor is zero: 
\begin{equation}
\VV_{0} =0 \, .
\label{V0cdbis}
\end{equation}
Then, the first term of the magnetic field in the non-magnetic  region solves Maxwell equations with perfectly insulating electric boundary conditions on $\partial\Omega_{+}= \Sigma \cup\Gamma$ (we recall $\kappa_{+}=\omega\sqrt{\varepsilon_{0}\mu_{+}}$, cf.  \eqref{2Eeps}):
\begin{equation*}
 \left\{
   \begin{array}{lll}
    \rot\rot  \HH^+_{0} - \kappa_{+}^2({1+\frac{i}{\delta_{+}^2}})  \HH^+_{0} =  \rot\jj    \quad&\mbox{in}\quad \Omega_{+}
\\[0.5ex]
\HH^+_{0}\times\nn=0 \quad &\mbox{on}\quad \Sigma\cup \Gamma .
   \end{array}
    \right.
\end{equation*}

The first profile in the magnetic region is a tangential field which is exponential with a complex rate  $\lambda$:
\begin{equation}
\label{V1cd}
\cV_{1}(y_{\alpha},Y_{3}) = -\sj_{0}(y_{\alpha})\;\mathrm{e}^{-\lambda Y_{3}}\, .
\end{equation}
Here $\sj_{0}(y_{\alpha})=   \lambda^{-1} ({1+\frac{i}{\delta_{-}^2}}) ({1+\frac{i}{\delta_{+}^2}})^{-1} \left( \rot \HH_{0}^+ \times\nn \right) (y_{\alpha},0)$,  
and $\lambda$ is given by: 
\begin{equation}
\label{lambda}
\lambda=\kappa_{+} \sqrt[4]{1+ \frac1{\delta_{-}^{4}}}\,\mathrm{e}^{\di i\tfrac{\theta(\delta_{-})-\pi}{2}} \quad \mbox{with}\quad  \delta_{-} =\sqrt{{\omega\varepsilon_0} /{\sigma_{-}}}  \ ,\quad \mbox{and}\quad  \theta(\delta_{-})=\arctan\frac1{\delta_{-}^{2}} \, .
\end{equation} 
Note that $\Re \lambda>0$, and if $\sj_{0}$ is not identically $0$, there exists a constant $C>0$ independent of $\eps$ such that
\begin{equation}
\label{Eestim}
   C^{-1}\sqrt \eps \leqslant
   \|\HH^-_1(\,\cdot\,;\eps) \|_{0,\Omega_{-}}\leqslant C\sqrt \eps \ .
\end{equation}    
Note also that the normal component of the profile $\VV_1$ is zero:  
\begin{equation*}
\sv_{1}=0\,.
\end{equation*}
The next term in the  non-magnetic 
region solves:
\begin{equation*}
 \left\{
   \begin{array}{lll}
    \rot\rot  \HH^+_{1} - \kappa_{+}^2({1+\frac{i}{\delta_{+}^2}})  \HH^+_{1} =  0   \quad&\mbox{in}\quad \Omega_{+}
\\[0.5ex]
\HH^+_{1}\times\nn= -\sj_{0} \times\nn \quad &\mbox{on}\quad \Sigma
\\[0.5ex]
  \HH^+_{1} \times \nn= 0   \quad &\mbox{on}\quad \Gamma .
   \end{array}
    \right.
\end{equation*}

 Like above, define $\sj_1$ as  $\sj_{1}(y_{\alpha})=   \lambda^{-1} ({1+\frac{i}{\delta_{-}^2}}) ({1+\frac{i}{\delta_{+}^2}})^{-1} \left( \rot \HH_{1}^+ \times\nn \right) (y_{\alpha},0)$ on the interface $\Sigma$. Then, the tangential components of the profile $\VV_{2}$ are given by the tangential field  $\cV_{2}$:
\begin{gather}
\label{V2cdbis}
   \cV_{2}(y_{\alpha},Y_{3})=
   \Big[ -\sj_{1} + \big(\lambda^{-1} + Y_{3}\big) \left(\cC-\cH\right)  \sj_{0} 
   \Big](y_{\alpha}) \;\,\mathrm{e}^{-\lambda Y_{3}}\, .
 \end{gather}
Here $\cH=\tfrac12\, b_{\alpha}^{\alpha}$ is the \textit{mean curvature} of the surface $\Sigma$\footnote{In particular, the sign of $\cH$ depends on the orientation of the surface $\Sigma$. As a convention, the unit normal vector $\nn$ on the surface $\Sigma$ is inwardly oriented to $\Omega_{-}$, see Figure~\ref{Fig1}. }, and $\cC$ is the curvature tensor field on  $\Sigma$ defined by 
\begin{equation}
\label{EcC}
(\cC  \sj)_{\alpha}=b^{\beta}_{\alpha}\, \sj_{\beta} \, ,
\end{equation}
with $b_{\alpha}^{\beta}=a^{\beta\gamma}b_{\gamma\alpha}$, and $a^{\beta\gamma}$ is the inverse of the metric tensor $a_{\beta\gamma}$ in $\Sigma$, and $b_{\gamma\alpha}$ is the curvature tensor in $\Sigma$.
The next term which is determined is the normal component $\sv_{2}$ of the profile $\VV_{2}$: 
\begin{equation*}
\sv_{2}(y_{\alpha},Y_{3})= 
   -\lambda^{-1}  \Div_{\Sigma} \ \sj_{0}(y_{\alpha})\,\mathrm{e}^{-\lambda Y_{3}} \, .
\end{equation*}
 Here $\Div_{\Sigma}$ is the surface divergence operator on $\Sigma$. The next term in the non-magnetic region solves :
\begin{equation*}
 \left\{
   \begin{array}{lll}
    \rot\rot  \HH^+_{2} - \kappa_{+}^2({1+\frac{i}{\delta_{+}^2}})  \HH^+_{2} =  0   \quad&\mbox{in}\quad \Omega_{+}
\\[0.9ex]
\nn\times \HH^+_{2}\times\nn= -\sj_{1}+ \lambda^{-1}\left(\cC-\cH\right)  \sj_{0} \quad &\mbox{on}\quad \Sigma
\\[0.5ex]
  \HH^+_{2} \times \nn= 0   \quad &\mbox{on}\quad \Gamma .
   \end{array}
    \right.
\end{equation*}

\subsection{Validation of the multiscale expansion}
\label{Svalid}

\newcommand  {\CC}{\boldsymbol{\mathsf C}}
\newcommand  {\jjw}{\widetilde{\boldsymbol{\mathsf\jmath}}}

The validation of the multiscale expansion \eqref{6E4abis}-\eqref{6E4bbis} for the magnetic field $\HH_{(\eps)}$ 
consist in proving estimates for remainders $\RR_{m;\,\eps}$ which are defined as 
\begin{equation}
\label{6E5}
   \RR_{m;\,\eps} = \HH_{(\eps)} - \sum_{j=0}^m \eps^j\HH_j \quad\mbox{in}\quad\Omega\,.
\end{equation}
This is done by an evaluation of the right hand side when the Maxwell operator is applied to $\RR_{m;\,\eps}$. By construction (cf. Sec. \ref{AH2}, Appendix \ref{AppA}), we obtain
\begin{equation}
\label{6E6}
\left\{
   \begin{array}{lll}
   \rot \alpha^{-1}_+ \rot\RR^+_{m;\,\eps} - \kappa_{+}^2\RR^+_{m;\,\eps} &=\quad 0 
   \quad & \mbox{in}\quad \Omega_+
   \\[0.5ex]
   \rot \alpha^{-1}_-\rot\RR^-_{m;\,\eps} - \eps^{-2}\kappa_{+}^2\RR^-_{m;\,\eps} &=\quad \jj^-_{m;\,\eps} 
   \quad & \mbox{in}\quad \Omega_-
   \\[0.5ex]
   \big[\RR_{m;\,\eps}\times\nn\big]_\Sigma &=\quad 0 \quad  &\mbox{on}\quad \Sigma
   \\[0.5ex]
   \big[\underline{\alpha}^{-1}\rot\RR_{m;\,\eps}\times\nn\big]_\Sigma &=\quad \g_{m;\,\eps}  
   \quad  &\mbox{on}\quad \Sigma
   \\[0.5ex]
   \RR^+_{m;\,\eps}\times\nn &=\quad 0 \quad &\mbox{on}\quad \partial\Omega\,.
   \end{array}
    \right.
\end{equation}
Here, $\underline{\alpha}=(\alpha_{+}, \alpha_{-})$, $\alpha_+=1+i/\delta_{+}^2$ and $\alpha_-=1+i/\delta_{-}^2$, and $[\HH\times\nn]_\Sigma$ denotes the jump of $\HH\times\nn$ across $\Sigma$. The right hand sides (residues) $\jj^-_{m;\,\eps}$ and $\g_{m;\,\eps}$ are, roughly, of the order $\eps^{m}$ :  we have $\jj^-_{0;\,\eps}=0$ and $\g_{0;\, \eps}=\mathcal{O}(1)$ ;  for all $m\in\N\setminus\{0\}$, we have $\jj^-_{m;\,\eps}=\mathcal{O}({\eps^{m-1}})$ and $\g_{m;\, \eps}=\mathcal{O}({\eps^{m}})$, and we have  the following estimates 
\begin{equation}
\label{6E7}
   \|\jj^-_{m;\,\eps}\|_{0,\Omega_-}   \leqslant C_m\eps^{m-1}     \quad \mbox{and} \quad 
   \|\g_{m;\,\eps}\|_{\frac12,\Sigma}   
\leqslant
   C_m\eps^{m},
\end{equation}
where $C_m>0$ is independent of $\eps$. The main result of this section is the following.

\begin{thm}
\label{6T1}
In the framework above, for all $m\in\N$ and $\eps\in(0,\eps_0]$, the remainder $\RR_{m;\,\eps}$ \eqref{6E5} satisfies the optimal estimate
\begin{equation}
\label{6E8}
   \|\RR^+_{m;\,\eps}\|_{0,\Omega_+} + \|\rot\RR^+_{m;\,\eps}\|_{0,\Omega_+} + 
   \eps^{-\frac12} \|\RR^-_{m;\,\eps}\|_{0,\Omega_-} + 
   \eps^{\frac12} \|\rot\RR^-_{m;\,\eps}\|_{0,\Omega_-} \leqslant C_m\eps^{m+1}.
   \!\!\!\!
\end{equation}
\end{thm}
\begin{proof}
{\sc Step 1.}  Since $\g_{m;\,\eps} \neq 0$, according to \eqref{6E6} the vector field $\underline{\alpha}^{-1}\rot\RR_{m;\,\eps}$ does not define an element of the Hilbert space $\bH(\rot,\Omega)$:
\begin{equation*}
\bH(\rot,\Omega) =\{\uu\in\bL^2(\Omega)\; |\ \rot\uu\in\bL^2(\Omega) \}\ .
\end{equation*}
Hence  we can not apply Theorem \ref{2T0} directly because the vector field $\rot\underline{\alpha}^{-1}\rot\RR_{m;\,\eps} - \kappa_{+}^2\mur(\eps)\RR_{m;\,\eps}$ does not define an element of $\bL^2(\Omega)$. 
We are going to introduce a corrector $\CC_{m;\,\eps}$ satisfying suitable estimates and so that
\begin{equation}
\label{6E6C}
\left\{
   \begin{array}{lll}
   \big[(\RR_{m;\,\eps} - \CC_{m;\,\eps})\times\nn\big]_\Sigma 
   &=\quad 0 \quad  &\mbox{on}\quad \Sigma
   \\[0.5ex]
   \big[\underline{\alpha}^{-1}\rot(\RR_{m;\,\eps}-\CC_{m;\,\eps})\times\nn\big]_\Sigma 
   &=\quad 0   \quad  &\mbox{on}\quad \Sigma
    \\[0.5ex]
   (\RR_{m;\,\eps}-\CC_{m;\,\eps})\times\nn  
   &=\quad 0   \quad  &\mbox{on}\quad \partial\Omega\ .
   \end{array}
    \right.
\end{equation}

\noindent
Construction of $\CC_{m;\,\eps}$: We take $\CC_{m;\,\eps}=0$ in $\Omega_-$ and use a trace lifting to define $\CC_{m;\,\eps}$ in $\Omega_+$. It suffices that
\begin{equation}
\label{6E7C}
\left\{
   \begin{array}{lll}
   \CC^+_{m;\,\eps} \times\nn 
   &=\quad 0 \quad  &\mbox{on}\quad \Sigma
   \\[0.5ex]
  {\alpha}_{+}^{-1} \rot\CC^+_{m;\,\eps} \times\nn 
   &=\quad \g_{m;\,\eps}   \quad  &\mbox{on}\quad \Sigma
     \\[0.5ex]
   \CC^+_{m;\,\eps} \times\nn 
   &=\quad  0   \quad  &\mbox{on}\quad \partial\Omega \ .
   \end{array}
    \right.
\end{equation}
Denoting by $C_\beta$ and $C_3$ the tangential and normal components of $\CC^+_{m;\,\eps}$ associated with a system of normal coordinates $\yy=(y_\beta,y_3)$, and by $g_\beta$ the components of $\g_{m;\,\eps}$ the above system becomes (cf.\ \cite[Sec. 5.2, Eq. (5.12)]{CDP09}) 
\begin{equation}
\label{6E8C}
\left\{
   \begin{array}{lll}
   C_\beta 
   &=\quad 0 \quad  &\mbox{on}\quad \Sigma
   \\[0.5ex]
   {\alpha}_{+}^{-1} ( \partial_3 C_\beta - \partial_\beta C_3 )
   &=\quad g_\beta   \quad  &\mbox{on}\quad \Sigma
    \\[0.5ex]
        C_\beta 
   &=\quad 0   \quad  &\mbox{on}\quad \partial\Omega \ .
   \end{array}
    \right.
\end{equation}
It can be solved in $\bH^2(\Omega_+)$ choosing $C_3=0$ and a standard lifting of the first two traces on $\Sigma$ and $\partial\Omega$ with the estimate
\begin{equation}
\label{6E9C}
   \|\CC^+_{m;\,\eps}\|_{2,\Omega_+} \leqslant C\|\g_{m;\,\eps}\|_{\frac12,\Sigma} \ .
\end{equation}

We deduce from assumption \eqref{6E7}, and \eqref{6E9C}
\begin{equation}
\label{6E11}
   \|\CC^+_{m;\,\eps}\|_{2,\Omega_+}
 \leqslant 
   C\eps^{m} ,
\end{equation}
where $C$ may depend on $m$.
We set
\begin{equation}
\label{6E12}
   \widetilde\RR_{m;\,\eps} := \RR_{m;\,\eps} - \CC_{m;\,\eps} 
   \quad\mbox{and}\quad
  \jjw_{m;\,\eps} := 
  \rot\underline{\alpha}^{-1}\rot\widetilde\RR_{m;\,\eps} - \kappa_{+}^2\mur(\eps)\widetilde\RR_{m;\,\eps} \ .
\end{equation}
Hence by construction, we have $\jjw_{m;\,\eps}\in\bL^{2}(\Omega)$ with the estimates
\begin{equation}
\label{6E13}
   \|\jjw_{m;\,\eps}\|_{0,\Omega}  \leqslant 
   C\eps^{m-1} .
\end{equation}

\noindent{\sc Step 2.}
We can apply Theorem \ref{2T0}  to the couple $(\HH,\EE) = (\widetilde\RR_{m;\,\eps},(\bs-i\omega\eps_0)^{-1}\rot\widetilde\RR_{m;\,\eps})$ and, thanks to \eqref{6E12}, we obtain 
\begin{equation*}
   \|\widetilde\RR_{m;\,\eps}\|_{0,\Omega} + \|\rot\widetilde\RR_{m;\,\eps}\|_{0,\Omega}
   \leqslant C \|\jjw_{m;\,\eps}\|_{0,\Omega}\,.
\end{equation*}
Combining this estimate with \eqref{6E11} and \eqref{6E13}, we deduce
\begin{equation}
\label{6E14}
   \|\RR_{m;\,\eps}\|_{0,\Omega} + \|\rot\RR_{m;\,\eps}\|_{0,\Omega}
   \leqslant  C\eps^{m-1},
\end{equation}
where $C$ may depend on $m$.

\noindent{\sc Step 3.}
In order to have an optimal estimate for $\RR_{m;\,\eps}$, we use \eqref{6E14} for $m+2$. We obtain
\begin{equation}
\label{6E15}
   \|\RR_{m+2;\,\eps}\|_{0,\Omega} + \|\rot\RR_{m+2;\,\eps}\|_{0,\Omega}
   \leqslant  C\eps^{m+1} \, .
\end{equation}
But we have the formula
\begin{equation}
\label{6E16}
   \RR_{m;\,\eps} = \sum_{j=m+1}^{m+2} \eps^j \HH_j + \RR_{m+2;\,\eps} \, ,
\end{equation}
and using  \eqref{6E4bbis}, we have also for any $j\in\N$
\begin{equation}
\label{6E17}
   \|\HH^+_j\|_{\bH(\rot,\Omega_+)} + 
   \eps^{-\frac12} \|\HH^-_j\|_{0,\Omega_-} + 
   \eps^{\frac12} \|\rot\HH^-_j\|_{0,\Omega_-} \leqslant C.
\end{equation}
We finally deduce the wanted estimate \eqref{6E8} from \eqref{6E15} to \eqref{6E17}.
\end{proof}

\section{Conclusion and perspectives}

In this work we tackled a transmission problem  in materials with high contrast in magnetic permeabilities. We proved uniform a priori estimates for the electromagnetic field solution of the time-harmonic Maxwell equations as the relative magnetic permeability $\mu_{r}$ between  a magnetic conductor and a non-magnetic material tends to infinity, and when the interface between the subdomains is Lipschitz. Assuming smoothness for the interface between the subdomains, we derived a multiscale expansion   for the magnetic field  in powers of  $\dfrac1{\sqrt{\mu_{r}}}$ with profiles inside the magnetic conductor, and we proved error estimates at any order.

One perspective would be to carry out numerical experiments in axisymmetric domains \cite{B-D-M99,HiLe05,Nk05,CDP09}. This would make it possible in particular to illustrate our theoretical results. The multiscale expansion provides also an a priori knowledge on the magnetic field useful for the design of meshes in view of a quality finite element approximation. Thus, the mesh should be matched to the boundary of the magnetic conductor  and can be coarse far from its boundary in the magnetic conductor, as with the  {\em skin effect} in non-magnetic conductive materials. It would be useful also to derive impedance boundary conditions (see, {\it e.g.},  \cite{artola1992diffraction,lafitte1993equations,HJN08} for scattering
problems from absorbing obstacles) on the interface (magnetic conductor / non-magnetic material)  in order  to reduce the computational costs.

\clearpage

\appendix

\section{Elements of proof for the multiscale expansion}
\label{AppA}

	In the framework of Section \ref{S3bis}, we assume that the surfaces $\Sigma$ and $\Gamma$ are smooth (cf. Assumption \ref{H2}). Recall that there exists  $\eps_{\star}>0$ and such that for all $\eps\in(0,\eps_{\star})$, the  problem \eqref{MS}-\eqref{PIbc} has a unique solution $(\EE_{(\eps)}, \,  \HH_{(\eps)})$ which is denoted by $(\EE^+_{(\eps)}, \,  \HH^+_{(\eps)})$ in $\Omega_{+}$, and by $(\EE^-_{(\eps)}, \,  \HH^-_{(\eps)})$ in $\Omega_{-}$.
Furthermore, we assume that the right hand side $\jj$ is smooth and its support does not  the domain $\Omega_{-}$.

Recall that $(y_{\alpha},y_{3})$ is a {\em normal coordinate system} to the surface $\Sigma$ in  $\cU_-$, see Figure~\ref{Fig1}. The function $\yy\mapsto\chi(y_3)$ is a smooth cut-off with support in $\overline\cU_-$ and equal to $1$ in a smaller tubular neighborhood of $\Sigma$.  
\begin{thm}
\label{ThA1}
Under the above assumptions, the magnetic field $\HH_{(\eps)}$  possesses the asymptotic expansion (see Theorem \ref{6T1} in Section \ref{Svalid} for precise estimates):  
\begin{gather}
\label{EAa}
     \HH^+_{(\eps)}(\xx) \approx \sum_{j\geqslant0} \eps^j\HH^+_j(\xx) \, ,
\\
\label{EAc}
   \HH^-_{(\eps)}(\xx) \approx \sum_{j\geqslant0} \eps^j\HH^-_j(\xx;\eps) 
   \quad\mbox{with}\quad   
   \HH^-_j(\xx;\eps) = \chi(y_3) \,\VV_j(y_\alpha,\frac{y_3}{\eps})\, ,
\end{gather}
where $\VV_j(y_\alpha,Y_{3})\rightarrow 0$ when $Y_{3}\rightarrow \infty$. Moreover, for any $j\in\N$, we have 
\begin{equation}
\label{EAd}
  \HH^+_j\in\bH(\rot,\Omega_+) \quad\mbox{and}\quad
  \VV_{j}\in\bH(\rot,\Sigma\times\R_+).
\end{equation}
\end{thm}

Hereafter, we present elements of proof of this theorem and details about the terms in asymptotics \eqref{EAa}--\eqref{EAc}. In \S\ref{AH1}, we expand the ``magnetic'' Maxwell operators in power series of $\eps$ inside the boundary layer $\cU_-$. We deduce in \S\ref{AH2} the equations satisfied by the magnetic profiles, and we derive explicitly the first ones in \S\ref{AH3}.

\subsection{Expansion of the Maxwell operators}
\label{AH1}

Integrating by parts in the variational formulation \eqref{FVH0}, we find the following Maxwell transmission problem for the magnetic field (we recall $ \kappa_{+}=\omega\sqrt{\varepsilon_{0}\mu_{+}}$, cf. \eqref{2Eeps})

\begin{equation}
\label{EMH}
 \left\{
   \begin{array}{lll}
   \rot\rot  \HH^+_{(\eps)} - \kappa_{+}^2({1+\frac{i}{\delta_{+}^2}}) \HH^+_{(\eps)} = \rot\jj   \quad&\mbox{in}\quad \Omega_{+}
\\[0.8ex]
\eps^{2}\rot\rot \HH^-_{(\eps)} -\kappa_{+}^2({1+\frac{i}{\delta_{-}^2}}) \HH^-_{(\eps)} =0
 \quad&\mbox{in}\quad \Omega_{-}
\\[0.8ex]
 ({1+\frac{i}{\delta_{+}^2}})^{-1} \rot\HH^+_{(\eps)}\times\nn=  ({1+\frac{i}{\delta_{-}^2}})^{-1}\rot\HH^-_{(\eps)}\times\nn \quad &\mbox{on} \quad \Sigma
\\[0.8ex]
\HH^+_{(\eps)}\times\nn=\HH^-_{(\eps)}\times\nn  \quad &\mbox{on}\quad \Sigma
\\[0.8ex]
  \HH^+_{(\eps)} \times\nn= 0   
  \quad &\mbox{on}\quad \Gamma .
   \end{array}
    \right.
\end{equation}
As a consequence of the Maxwell transmission problem, we have also  
\[
   \Div\bm\HH_{(\eps)} = 0 \quad\mbox{in}\quad\Omega.
\]
Therefore, we have in particular $\bm\HH_{(\eps)} \in\bH(\Div,\Omega)$, where 
\begin{equation*}
\bH(\Div,\Omega) =\{\uu\in\bL^2(\Omega)\; |\ \Div\uu\in\L^2(\Omega) \}\ ,
\end{equation*}
and from which we deduce the extra transmission condition
\begin{equation}
\label{AHn}
 \eps^{2}\HH^+_{(\eps)}\cdot\nn=\HH^-_{(\eps)}\cdot\nn  \quad \mbox{on}\quad \Sigma.
\end{equation}

In that follows we denote by $\LL(y_{\alpha}, h;D_{\alpha}, \partial_{3}^h)$ the second order Maxwell operator $\eps^{2}\rot\rot-\kappa_{+}^2({1+\frac{i}{\delta_{-}^2}})\Id$ set in $\cU_{-}$ in a {\em normal coordinate system} $(y_{\alpha},h)$. Here $D_{\alpha}$ is the covariant derivative on the interface $\Sigma$, and $\partial_{3}^h$ is  the partial derivative with respect to the normal coordinate $y_{3}=h$. 
 Then after the scaling $Y_{3}=\eps^{-1} h$,  the operator which is written $\LL[\eps]$, expands in power series of $\eps$ with coefficients intrinsic operators :
\begin{equation*}
 \LL[\eps]=\di\sum_{n=0}^{\infty} \eps^n\LL^{n} 
  \ .
\end{equation*}
We denote by $L_{\alpha}^n$ the surface components of $\LL^{n}$. Using a splitting  a three-dimensional vector field $\VV$ into its normal component $\sv$ and its tangential component $\cV$ that can be viewed as a one-form field $\cV_{\alpha}$, and using the summation convention of repeated two dimensional indices (represented by greek letters), we have
\begin{equation}
L_{\alpha}^0(\VV)=-\partial_{3}^2 \cV_{\alpha}-\kappa_{+}^2({1+\frac{i}{\delta_{-}^2}}) \cV_{\alpha} 
  \ \  \mbox{and}  \ \ 
L_{\alpha}^1(\VV)=-2b_{\alpha}^\beta\partial_{3} \cV_{\beta} 
+ \partial_{3} \partial_{\alpha} \sv 
+  b_{\beta}^{\beta}\partial_{3}\cV_{\alpha}\ ,
\label{EHsL0L1}
\end{equation}
were $\partial_{3}$ is the partial derivative with respect to $Y_{3}$ and  $b_{\alpha}^{\beta}=a^{\beta\gamma}b_{\gamma\alpha}$. Here $a^{\beta\gamma}$ is the inverse of the metric tensor $a_{\beta\gamma}$ in $\Sigma$, and $b_{\gamma\alpha}$ is the curvature tensor in $\Sigma$. Then we denote by $L_{3}^n$ the transverse components of $\LL^{n}$. We have
\begin{equation}
L_{3}^0(\VV)=-\kappa_{+}^2 ({1+\frac{i}{\delta_{-}^2}})\sv 
  \quad \mbox{and}  \quad
L_{3}^1(\VV)=\gamma_{\alpha}^{\alpha}(\partial_{3}\VV)+  b_{\beta}^{\beta}\partial_{3} \sv\ .
\label{EHtL0L1}
\end{equation}
Here $\gamma_{\alpha\beta}(\VV)=\frac12 (D_{\alpha}\cV_\beta + D_{\beta}\cV_\alpha)-b_{\alpha\beta}\sv$  is the change of metric tensor.

Then we denote by $\BB(y_{\alpha}, h;\partial_{\alpha}, \partial_{3}^h)$ the tangent trace operator $\rot\cdot\times\nn$ on $\Sigma$ in a {\em normal coordinate system}. If $\VV=(\cV_{\alpha},\sv)$,  we have (cf. \cite[Sec. A.4]{CDFP11})
\begin{equation*}
\left( \BB(y_{\alpha}, h;\partial_{\alpha} , \partial_{3}^h)\VV\right)_{\alpha}=\partial_{3}^h \cV_{\alpha} - \partial_{\alpha} \sv \ .
\end{equation*} 
Then if $\BB[\eps]$ is the operator  obtained from $\BB$ in $\cU_{-}$ after the scaling $Y_{3}=\eps^{-1} h$, we have
\begin{equation*}
 \BB[\eps]=\eps^{-1}\BB^{0}+\BB^{1} \ .
\end{equation*}
Finally we denote by $B_{\alpha}^n$ the surface components of $\BB^{n}$.  We have
\begin{equation}
\label{EB0B1}
B_{\alpha}^0(\VV)= \partial_{3}\cV_{\alpha}   \quad \mbox{and}  \quad B_{\alpha}^1(\VV)= -\partial_{\alpha} \sv \ .
\end{equation}

\subsection{Equations for the coefficients of the magnetic field}
\label{AH2}

According to the second and third equations in system \eqref{EMH}, the profiles $\VV_{j}$ and the coefficients $\HH_{j}^+$ of the magnetic field satisfy the following system
\begin{gather}
\label{ELV}
 \LL[\eps] \sum_{j\geqslant0} \eps^j\VV_j(y_{\alpha},Y_{3})=0  \quad \mbox{in}\quad \Sigma \times I,
 \\
\label{ETV}
\BB[\eps] \sum_{j\geqslant0} \eps^j\VV_j(y_{\alpha},0)=({1+\frac{i}{\delta_{-}^2}}) ({1+\frac{i}{\delta_{+}^2}})^{-1}\sum_{j\geqslant0} \eps^j \rot \HH_{j}^+ \times\nn \quad \mbox{on} \quad \Sigma \ ,
\end{gather}
with $I = (0,+\infty)$. 

Then we perform in \eqref{ELV}-\eqref{ETV} the identification of terms with the same power in $\eps$. The components of the equation \eqref{ELV} are the collections of equations
\begin{equation}
\label{EsLVj}
\LL^0(\VV_{0})=0\ , \quad \LL^0(\VV_{1})+\LL^1(\VV_{0})=0 \ , \quad \mbox{and} \quad \di\sum_{l=0}^{n} \LL^{n-l}(\VV_{l})=0  \ ,
\end{equation}
for all $n\geqslant2$. The surface components of the equation \eqref{ETV} write 
\begin{equation}
\label{EsTVj}
 B_{\alpha}^0(\VV_{0}) =0 \quad
\mbox{and} \quad B_{\alpha}^0(\VV_{n})+B_{\alpha}^1(\VV_{n-1}) =  ({1+\frac{i}{\delta_{-}^2}}) ({1+\frac{i}{\delta_{+}^2}})^{-1}\left(\rot \HH_{n-1}^+ \times\nn \right)_{\alpha} \ ,
\end{equation}
for all $n\geqslant1$.

Using the components of the operator $\LL^0$ \eqref{EHsL0L1}-\eqref{EHtL0L1}, and  the surface components of $\BB_{0}$ and $\BB_{1}$ \eqref{EB0B1}, 
 and expanding $\HH^+_{(\eps)}$ in $\Omega_+$, we thus see that, according to the system \eqref{EMH}, the profiles $\VV_n = (\cV_{n},\sv_n)$ and the terms $\HH^+_{n}$ have to satisfy, for all $n \geq 0$, 
 \begin{equation}
\label{EH}
 \left\{
   \begin{array}{clll}
      (i) & \quad &
   \partial_{3}^2 \cV_{n,\alpha}-\lambda^2 \cV_{n,\alpha} =  \di \sum_{j = 0}^{n-1} L_\alpha^{n-j}(\VV_j) 
 \quad&\mbox{in}\quad \Sigma \times I
\\[0.8ex]
 (ii) & \quad &
     \partial_{3}\cV_{n,\alpha}  = \partial_{\alpha} \sv_{n-1}+ ({1+\frac{i}{\delta_{-}^2}}) ({1+\frac{i}{\delta_{+}^2}})^{-1} \left( \rot \HH_{n-1}^+ \times\nn \right)_{\alpha}
  \quad&\mbox{on}\quad \Sigma
 \\[0.8ex]
    (iii) & \quad &
   -\lambda^2 \sv_n =  \di \sum_{j = 0}^{n-1} L_3^{n-j}( \VV_j) 
 \quad&\mbox{in}\quad \Sigma \times I
\\[0.8ex]
     (iv) & \quad &  
   \rot\rot  \HH^+_{n} - \kappa_{+}^2({1+\frac{i}{\delta_{+}^2}})  \HH^+_{n} = \delta_n^0 \rot\jj   \quad&\mbox{in}\quad \Omega_{+}
\\[0.8ex]
     (v) & \quad &
\nn\times  \HH^+_{n}\times\nn  =  \cV_{n} &\mbox{on} \quad \Sigma   
\\[0.8ex]
      (vi) & \quad &
  \HH^+_{n} \times\nn= 0 \quad   &\mbox{on}\quad \Gamma  \ .
   \end{array}
    \right.
\end{equation}
Here we have 
\begin{equation*}
\lambda=\kappa_{+} \sqrt[4]{1+ \frac1{\delta_{-}^{4}}}\,\mathrm{e}^{i\frac{\theta(\delta_{-})-\pi}{2}} \quad \mbox{with} \quad \theta(\delta_{-})=\arctan\frac1{\delta_{-}^{2}}\ ,
 \end{equation*}
 so that:  $-\lambda^2=\kappa_{+}^2({1+\frac{i}{\delta_{-}^2}})$ and $\Re \lambda>0$.

Note that the transmission condition \eqref{AHn} implies the extra  condition 
\begin{equation}
\label{EH1}
 \sv_n =   \HH^+_{n-2}\cdot\nn \quad \mbox{on}\quad \Sigma. 
\end{equation}
The set of equations \eqref{EH}  allows to determine $\VV_n $ and $\HH^+_{n}$ by induction.

\subsection{First terms of the magnetic field asymptotics}
\label{AH3}
According to equations $(i)$-$(ii)$ in system \eqref{EH} for $n=0$, $\cV_{0}$ satisfies the following ODE 
\begin{equation}
\left\{
   \begin{array}{ll}
\partial_{3}^2\cV_{0}(.,Y_{3}) -\lambda^2 \cV_{0}(.,Y_{3})&=0 \quad \mbox{for} \quad Y_{3}\in I \ ,
\\[0.5ex]
\partial_{3}\cV_{0}(.,0)&=0 \ .
\end{array}
  \right.
\label{EV0}
\end{equation}
The unique solution of \eqref{EV0} such that $\cV_{0}\rightarrow 0$ when $Y_{3}\rightarrow\infty$, is $\cV_{0}=0$. From equation $(iii)$ in system \eqref{EH}  for $n=0$, there holds $\sv_{0}=0$. We infer 
\begin{equation}
\VV_{0}(y_{\beta},Y_{3}) =0 \, .
\label{V0cd}
\end{equation}
From equations $(iv)$-$(vi)$ in \eqref{EH} for $n=0$, and from \eqref{V0cd}, the first asymptotic of the magnetic field in the non-magnetic region solves the following problem 
\begin{equation*}
 \left\{
   \begin{array}{lll}
    \rot\rot  \HH^+_{0} - \kappa_{+}^2({1+\frac{i}{\delta_{+}^2}})  \HH^+_{0} =  \rot\jj    \quad&\mbox{in}\quad \Omega_{+}
\\[0.5ex]
\HH^+_{0}\times\nn=0 \quad &\mbox{on}\quad \Sigma\cup \Gamma .
   \end{array}
    \right.
\end{equation*}

The next term determined in the asymptotic expansion is $\cV_{1}$. From equations $(i)$-$(ii)$ in \eqref{EH}  for $n=1$, $\cV_{1}$ satisfies for $Y_3\in I$
\begin{equation*}
\left\{
   \begin{array}{lll}
\partial_{3}^2\cV_{1}(.,Y_{3}) -\lambda^2 \cV_{1}(.,Y_{3})&=0 \ ,
\\[0.5ex]
\partial_{3}\cV_{1}(.,0)&=  ({1+\frac{i}{\delta_{-}^2}}) ({1+\frac{i}{\delta_{+}^2}})^{-1}\left( \rot \HH_{0}^+ \times\nn \right)(.\,,0)    \ .
\end{array}
  \right.
\end{equation*}
We denote by $\jj_{k}(y_{\beta})=   \lambda^{-1} ({1+\frac{i}{\delta_{-}^2}}) ({1+\frac{i}{\delta_{+}^2}})^{-1} \left( \rot \HH_{k}^+ \times\nn \right) (y_{\beta},0)$ for $k=0,1$. Hence,
\begin{equation}
\label{V1cd}
\cV_{1}(y_{\beta},Y_{3}) = -\jj_{0}(y_{\beta})\;\mathrm{e}^{-\lambda Y_{3}}
\, .
\end{equation}
From equation $(iii)$ in \eqref{EH}  for $n=1$, and from \eqref{V0cd}, we obtain $\sv_{1}=0$. From equations $(iv)$-$(vi)$ in \eqref{EH} for $n=1$, and from \eqref{V1cd}, the asymptotic of order $1$ for the magnetic field in the non-magnetic part solves :
\begin{equation*}
 \left\{
   \begin{array}{lll}
    \rot\rot  \HH^+_{1} - \kappa_{+}^2({1+\frac{i}{\delta_{+}^2}})  \HH^+_{1} =  0   \quad&\mbox{in}\quad \Omega_{+}
\\[0.5ex]
\HH^+_{1}\times\nn= -\jj_{0} \times\nn \quad &\mbox{on}\quad \Sigma
\\[0.5ex]
  \HH^+_{1} \times \nn= 0   \quad &\mbox{on}\quad \Gamma .
   \end{array}
    \right.
\end{equation*}

Then, from equations $(i)$-$(ii)$ in \eqref{EH} for $n=2$, $\cV_{2,\alpha}$ solves the ODE for $Y_3\in I$:
\begin{equation*}
\left\{
   \begin{array}{lll}
   \partial_{3}^2\cV_{2,\alpha}(.\,,Y_{3}) -\lambda^2 \cV_{2,\alpha}(.\,,Y_{3}) & = -2b_{\alpha}^{\sigma}\partial_{3}\cV_{1,\sigma}(.\,,Y_{3})+b_{\beta}^{\beta}\partial_{3}\cV_{1,\alpha}(.\,,Y_{3}) 
\\[0.8ex]
\partial_{3}\cV_{2,\alpha}(.\,,0) & =  ({1+\frac{i}{\delta_{-}^2}}) ({1+\frac{i}{\delta_{+}^2}})^{-1}\left( \rot \HH_{1}^+ \times\nn \right)_{\alpha}(.\,,0)\ .
   \end{array}
\right.
\end{equation*}
We denote by $\sj_{k,\alpha}$ the surface components of $\jj_{k}$, for $k=0,1$. We obtain
\begin{gather}
\label{V2cd}
   \cV_{2,\alpha}(y_{\beta},Y_{3})=
   \Big[ -\sj_{1,\alpha} + \big(\lambda^{-1} + Y_{3}\big) 
   \big( b_{\alpha}^{\sigma}\ \sj_{0,\sigma}-\cH\,\sj_{0,\alpha}\big) 
   \Big](y_{\beta}) \;\,\mathrm{e}^{-\lambda Y_{3}}\, .
 \end{gather}
From equation $(iii)$ in \eqref{EH} for $n=2$, we infer
\begin{equation*}
\sv_{2}(y_{\beta},Y_{3})= 
   -\lambda^{-1} D_{\alpha}\ \sj_{0}^{\alpha}(y_{\beta})\,\mathrm{e}^{-\lambda Y_{3}}  .
\end{equation*}
From equations $(iv)$-$(vi)$ in \eqref{EH} for $n=2$, and from \eqref{V2cd}, the asymptotic of order $2$ for the magnetic field in the non-magnetic part solves :
\begin{equation*}
 \left\{
   \begin{array}{lll}
    \rot\rot  \HH^+_{2} - \kappa_{+}^2({1+\frac{i}{\delta_{+}^2}})  \HH^+_{2} =  0   \quad&\mbox{in}\quad \Omega_{+}
\\[0.9ex]
\nn\times \HH^+_{2}\times\nn= -\jj_{1}+ \lambda^{-1}\left(\cC-\cH\right)  \jj_{0} \quad &\mbox{on}\quad \Sigma
\\[0.5ex]
  \HH^+_{2} \times \nn= 0   \quad &\mbox{on}\quad \Gamma .
   \end{array}
    \right.
\end{equation*}

\newpage
\bibliographystyle{plain}
\bibliography{biblio}

\end{document}